\documentclass[11pt]{article}

\usepackage{fullpage}
\usepackage{amsmath,amsthm,amsfonts}
\usepackage{amssymb,latexsym,graphicx}
\usepackage{fourier}
\usepackage{mathtools}
\usepackage{hyperref, cleveref}
\usepackage{authblk}
\usepackage{graphicx, xcolor}
\usepackage{url}
\usepackage[ruled]{algorithm2e}

\makeatletter
\def\resetMathstrut@{%
  \setbox\z@\hbox{%
    \mathchardef\@tempa\mathcode`\(\relax
    \def\@tempb##1"##2##3{\the\textfont"##3\char"}%
    \expandafter\@tempb\meaning\@tempa \relax
  }%
  \ht\Mathstrutbox@1.2\ht\z@ \dp\Mathstrutbox@1.2\dp\z@
}
\makeatother
\newtheorem{theorem}{Theorem}[section]

\newtheorem{corollary}[theorem]{Corollary}
\newtheorem{definition}[theorem]{Definition}

\newcommand{\R}{\ensuremath{\mathbb{R}}}

\newcommand{\eps}{\varepsilon}
\renewcommand{\epsilon}{\varepsilon}

\renewcommand{\leq}{\leqslant}
\renewcommand{\geq}{\geqslant}

\newcommand{\E}{\mathbb{E}}

  \newcommand{\st}{:\,} % "such that" to define sets

\begin{document}

\title{Towards a Banach Space Chernoff Bound for Markov Chains via Chaining Arguments}
\author[1]{Shravas Rao \footnote{Funding: This material is based upon work supported by the National Science Foundation under Award No. 2348489}}
\affil[1]{Department of Computer Science, Portland State University \protect \\
1900 SW 4th Ave, Portland, OR 97201 \protect \\
shravas@pdx.edu}
\maketitle

\begin{abstract}
Let $\{Y_i\}_{i=1}^{\infty}$ be a stationary reversible Markov chain with state space $[N]$, let $(X, \| \cdot \|)$ be a real-valued Banach space and let $f_1, \ldots, f_n: [N] \rightarrow X$ be functions with mean $0$ such that $\|f_i(v)\| \leq 1$ for all $i$ and $v$.
We prove bounds on the expected value of and deviation bounds for the random variable  $\|f_1(Y_1)+\cdots+f_n(Y_n)\|$.
For large enough $n$ that depends on the Banach space (and not $N$), these bounds behave similarly as known bounds for independent random variables.
When the Banach space in question is the set of matrices equipped with the $\ell_2 \rightarrow \ell_2$ operator norm, for large enough $n$, our bounds on the expected value improve upon known bounds and match what is known for independent random variables up to a factor in the spectral gap.
\end{abstract}

\section{Introduction}

It is well-known via the Chernoff Bound that the sum of independent bounded random scalars concentrates around its mean.
Given $n$ samples, the probability that the sum differs by more than $u \sqrt{n}$ than the expected value can be bounded above by $2 \exp(-C u^2)$ for some constant $C$. The Chernoff bound and its many variations, including the Hoeffding bound, Azuma inequality, Bernstein's inequality, and more, are used throughout theoretical computer science, often in the analysis of random structures and algorithms (see~\cite{BLM13} for more background).

Starting with the work of Rudelson~\cite{R99}, Ahlswede and Winter~\cite{AW06}, and Tropp~\cite{T12}, generalizations of these concentration inequalities were obtained for random matrices, rather than random scalars.
In particular, consider the operator norm of the sum of $n$ independently chosen random $d \times d$ matrices with mean $0$ and bounded norm.
It was shown that the probability that this exceeds $u \sqrt{n}$ was bounded above by $2 d \exp(-C u^2)$ for some constant $C$.
These bounds can be obtained in a variety of ways, including via the Golden-Thompson inequality~\cite{G65, T65} or Lieb's inequality~\cite{L73}.

Finally in work by Naor~\cite{N12}, building upon insights due to Pisier~\cite{P75}, it was shown that concentration occurs for all Banach spaces that satisfy certain properties.
Before stating the bound, we define a quantity related to Banach spaces called the modulus of uniform smoothness.

\begin{definition}
The modulus of uniform smoothness of a Banach space $(X, \| \cdot \|)$ is \[\rho_X(\tau) = \sup\left\{\frac{\|x+\tau y\|+\|x - \tau y\|}{2}-1 : x, y \in X, \|x\| = \|y\| = 1\right\}.\]
\end{definition}

Naor proved the following.

\begin{theorem}[Theorem,~\cite{N12}]\label{thm:naor}
Let $(X, \|\cdot\|)$ be a Banach space so that $\rho_X(\tau) \leq s \tau^2$ for some $s$ and all $\tau > 0$.
Let $f_1, \ldots, f_n: [N] \rightarrow X$ be functions such that $\E[f(Y_i)] = 0$ for all $i$ and $\|f_i(v)\| \leq 1$ for all $i$ and $v \in [N]$.
Let $Y_1, \ldots, Y_n$ be independent over some distribution $\mu$ on $[N]$
Then,
\begin{equation}\label{thm:banachspacech}
\Pr\left[\|f_1(Y_1)+\cdots+f_n(Y_n)\| \geq u \sqrt{n}\right] \leq \exp\left(s+2-{Cu^2}\right)
\end{equation}
for some universal constant $C$.
\end{theorem}

From Eq.~\eqref{thm:banachspacech}, one also has that $\E\left[\|f_1(Y_1)+\cdots+f_n(Y_n)\|\right] \leq C\sqrt{n s}$ for some constant $C$.
The modulus of uniform smoothness of $d \times d$ matrices under the operator norm is $O(\log(d))$ (see~\cite{TJ74}), and thus Naor's bound generalizes the matrix Chernoff bound.
\\

A different line of work aimed to generalize Chernoff bounds to the case of dependent random variables, in particular those obtained from a Markov chain with a spectral gap.
The main application was as a tool for derandomization.
Given a graph $G = (V, E)$ one can consider the random walk that starts at a uniformly random vertex, and chooses a random neighbor at each step.
If all vertices of the graph have low degree, this can use significantly less randomness than choosing vertices independently at random.

Gillman showed that if the transition matrix of the random walk had a small second largest value of an eigenvalue, then one can obtain similar concentration as promised by Chernoff bounds for independent random variables~\cite{G98}.
In particular, if $\lambda$ is the second largest eigenvalue of the transition matrix, one obtains the tail bound $2\exp(-C (1-\lambda)u^2)$ for some constant $C$.
Such bounds were refined in a long series of work, including~\cite{L98,LCP04,H08,ChungLLM12,HH15,R19,NRR20,M20,FJS21}.

There exist families of graphs of constant degree such that the second largest eigenvalue, and even second largest absolute value of an eigenvalue, is bounded above by a constant~\cite{LPS88,M88}.
Thus, what is remarkable about Chernoff bounds for Markov chains is that the same bound applies to graphs on an arbitrarily large number of vertices.
Additionally, one can obtain the same concentration while reducing the number of random bits from $n \log(|V|)$ to $\log(|V|)+O(n)$. \\

A natural next step is to obtain Chernoff bounds for vector-valued random variables for Markov chains.
In the case of matrices under the operator norm, such a bound was conjectured by Wigderson and Xiao, who also outlined a number of applications of such a bound~\cite{WX08}.
In particular, it was conjectured that the bound of $2 d \exp(-Cu^2)$ could be replaced by $2 d \exp(-C(1-\lambda)u^2)$.
This conjecture was proven true in work due to Garg, Lee, Song and Srivastava, due to a generalization of the Golden-Thompson inequality~\cite{GLSS18}.
This bound was further refined in~\cite{QWLPT20}

In the case of independent random matrices, one can add an extra $\left\|\E\left[X_i^2\right]\right\|$ factor to the bound, resulting in sharper bounds if this quantity is small.
This, and other variations of matrix Chernoff bounds were generalized to random matrices from a Markov chain in work due to Neeman, Shi and Ward~\cite{NSW24}.

In the case of $\ell_p$ spaces, a vector-valued Chernoff bound for Markov chains was obtained in~\cite{NRR20}.

In this work, we prove the following bound that works for all real-valued Banach spaces, and is independent of the size $N$ of the state space.

\begin{theorem}\label{thm:mainsame}
Let $\{Y_i\}_{i=1}^{\infty}$ be a stationary reversible Markov chain with state space $[N]$, transition matrix $A$ and stationary measure $\mu$ so that $Y_1$ is distributed according to $\mu$.
Let $E_{\mu} = \mathbf{1} \mu^*$ be the averaging operator, and let $\lambda = \left\|A-E_{\mu}\right\|_{L_2(\mu) \rightarrow L_2(\mu)}$ be less than $1$.

Let $g_1, \ldots, g_N$ be Gaussians with mean $0$ and such that $g_v$ has variance $\mu_v$.
\footnote{In particular, the variance of $g_v$ is equal to the probability of $v$ in the stationary distribution.}
For each $g_v$, let $g_v^{(1)}, \ldots, g_v^{(n)}$ be $n$ independent copies of $g_i$.
Let $(X, \| \cdot \|)$ be a $k$-dimensional Banach space over the real numbers, and let $f_1, \ldots, f_n: [N] \rightarrow X$ be functions such that $\E[f_i(Y_i)] = 0$ for all $i$ and $\|f_i(v)\| \leq 1$ for all $i$ and all $v \in [N]$.
Then for some constant $C$,
\[
\E\left[\|f_1(Y_1)+\cdots+f_n(Y_n)\|\right] \leq \frac{C k}{1-\lambda}+\frac{C}{\sqrt{1-\lambda}} \E\left[\left\|\sum_{i=1}^{n} \sum_{v=1}^N g_v^{(i)} f_i(v)\right\|\right] 
\]
\end{theorem}

Using the same techniques, we can obtain the following tail bound.

\begin{theorem}\label{cor:mstail}
In the setting of Theorem~\ref{thm:mainsame} and letting
\[
L = \E\left[\left\|\sum_{i=1}^{n} \sum_{v=1}^N g_v^{(i)} f_i(v)\right\|\right] ,
\]
then for some constants $C_1, C_2$
\[
\Pr\left[\|f_1(Y_1)+\cdots+f_n(Y_n)\| \geq u \right] \leq C_1 \exp\left(-C_2 (1-\lambda)\min\left\{\frac{u}{k}, \frac{u^2}{L^2}\right\}\right)
\]
\end{theorem}

Both bounds follow from a relatively straightforward application of techniques from generic chaining (see~\cite{T21}), and a Markov chain Bernstein inequality for scalar random variables due to Paulin~\cite{P15}.
This is in contrast to techniques used in prior work, which bounded the moment generating function or moments to obtain tail bounds.

To compare the bound in Theorem~\ref{thm:mainsame} with that of Theorem~\ref{thm:naor}, recall that the distribution of the normalized sum of Rademacher random variables approaches that of a standard Gaussian in the limit.
Thus, by Theorem~\ref{thm:naor}, 
\[
\E\left[\left\|\sum_{i=1}^{n} \sum_{v=1}^N g_v^{(i)} f_i(v)\right\|\right] 
\leq
\sqrt{n s},
\]
and when $n \geq O(k^2/(s(1-\lambda)))$, one obtains the same bound as what follows from Theorem~\ref{thm:naor} up to constant factors and the expected $1/\sqrt{1-\lambda}$ factor.
In fact, Theorem~\ref{thm:mainsame} can sometimes yield sharper bounds if one has a sharper bound on $\E\left[\left\|\sum_{i=1}^{n} \sum_{v=1}^N g_v^{(i)} f_i(v)\right\|\right].$

\subsection{Matrix-valued random variables}

We instantiate the bound in Theorem~\ref{thm:mainsame} to some concrete settings.

\paragraph*{Operator norm:} We can compare Theorem~\ref{thm:mainsame} to known results in the setting of symmetric matrix-valued random variables under the $\ell_2 \rightarrow \ell_2$ operator norm.
Here, $k$ will be the square of the dimension $d$ of the matrix.
%Assume that $\mu$ is uniform over $[N]$.

We have the following tail bound and bound on the expected value of the operator norm of a random Gaussian sum of matrices due to~\cite{O10} (see also~\cite{T12}).
\begin{theorem}\label{thm:indmat}
Let $A_1, \ldots, A_n \in \R^{d \times d}$ be symmetric matrices, and let $g_1, \ldots, g_n$ be independent Gaussians with mean $0$ and standard deviation $1$.
Then
%\[
%\Pr\left[\left\|\sum_{i=1}^n g_iA_i\right\| \geq u\left\|\sum_{i=1}^n A_i^2 \right\| \right]
%\leq 
%k \exp(-u^2/2)
%\]
%and 
for some universal constant $C$,
\[
E\left[\left\|\sum_{i=1}^n g_iA_i\right\| \right]
\leq 
C\sqrt{ \left\|\sum_{i=1}^n A_i^2 \right\| \log(d)}.
\]
\end{theorem}
Thus, the bound in Theorem~\ref{thm:mainsame} becomes
\[
\frac{C d^2}{1-\lambda}+\frac{C}{\sqrt{1-\lambda}} \sqrt{\left\|\E_{\mu}[f_1(v)^2+\cdots+f_n(v)^2]\right\|\log(d)}.
\]
We note that in Section~\ref{sec:sharp}, we explain that one can replace the $Cd^2$ term with $Cd$ without too much extra effort.

Neeman, Shi, and Ward~\cite{NSW24} obtain the bound
\[
\Pr\left[\left\|f_1(Y_1)+\cdots+f_n(Y_n)\right\| \geq u\right] \leq 2 d^{2-\pi/4} \exp\left(\frac{-u^2 \pi^2 /32}{\alpha(\lambda) \sigma^2 + \beta(\lambda) u}\right)
\]
where $\sigma^2 = \left\|\E_{\mu}[f_1(Y_1)^2]\right\|+\cdots+\left\|\E_{\mu}[f_n(Y_n)^2]\right\|$, $\lambda$ is the same as in Theorem~\ref{thm:mainsame} and
\[
\alpha(\lambda) = (1+\lambda)/(1-\lambda) \text{ and } \beta(\lambda) = \begin{cases} \frac{4}{3\pi} & \lambda = 0 \\ \frac{8/\pi}{1-\lambda} & 0 < \lambda < 1.\end{cases}
\]
Thus, when $n \geq Cd^4/((1-\lambda)\log(d))$, Theorem~\ref{thm:mainsame} can sometimes yield a sharper bound on the expected value.

In the case of independent random matrices, one has the bound due to~\cite{O09, T12}
\[
\Pr\left[\left\|f_1(Y_1)+\cdots+f_n(Y_n)\right\| \geq u\right] \leq 2 d  \exp\left(\frac{-u^2/2}{\sigma^2+u/3}\right)
\]
where this time, $\sigma^2 = \left\|\E_{\mu}[f_1(v)^2+\cdots+f_n(v)^2]\right\|$.
In particular, both the bound in Theorem~\ref{thm:mainsame} and known bounds for the independent case have the same $\sqrt{\left\|\E_{\mu}[f_1(v)^2+\cdots+f_n(v)^2]\right\|}$ term.

\paragraph*{Schatten norm:} We can also instantiate the bound in Theorem~\ref{thm:mainsame} to the setting of matrices under the Schatten-$p$ norm.
Given a $d \times d$ matrix $A$ with singular values $s_1(A), \ldots, s_d(A)$, the Schatten-$p$ norm of $A$, denoted as $S_p(A)$, is defined to by $(s_1(A)^p+\cdots+s_d(A)^p)^{1/p}$.
To our knowledge, there are no known results in this setting.

One has that $\rho_{S_p}(\tau) \leq \frac{p-1}{2}\tau^2$ from~\cite{TJ74,BCL94}.
Using the fact that the normalized sum of Rademacher random variables approaches that of a standard Gaussian in the limit and Theorem~\ref{thm:naor}, one has
\[
\E\left[\left\|\sum_{i=1}^{n} \sum_{v=1}^N g_v^{(i)} f_i(v)\right\|\right] 
\leq
\sqrt{\frac{p-1}{2}}\sqrt{\E_{\mu}[\|f_1(Y_1)\|^2]+\cdots+\E_{\mu}[\|f_n(Y_n)\|^2]},
\]
and Theorem~\ref{thm:mainsame} implies that
\[
\E\left[\|f_1(Y_1)+\cdots+f_n(Y_n)\|\right] \leq \frac{C d^2}{1-\lambda}+\frac{C}{\sqrt{1-\lambda}} \sqrt{\frac{p-1}{2}}\sqrt{\E_{\mu}[\|f_1(Y_1)\|^2]+\cdots+\E_{\mu}[\|f_n(Y_n)\|^2]}.
\]

\subsection{Future work}

Ideally, one would be able to obtain the bounds
\[
\E\left[\|f_1(Y_1)+\cdots+f_n(Y_n)\|\right] \leq \frac{C}{\sqrt{1-\lambda}} \E\left[\left\|\sum_{i=1}^{n} \sum_{v=1}^N g_v^{(i)} f_i(v)\right\|\right] 
\]
and
\[
\Pr\left[\|f_1(Y_1)+\cdots+f_n(Y_n)\| \geq u \E\left[\left\|\sum_{i=1}^{n} \sum_{v=1}^N g_v^{(i)} f_i(v)\right\|\right]  \right] \leq C_1 \exp\left(-C_2 (1-\lambda)u^2\right),
\]
in Theorems~\ref{thm:mainsame} and~\ref{cor:mstail} respectively, which we leave to future work.
It would also be interesting to develop a suitable bound for general state space Markov chains, and to obtain bounds that depend on only the second eigenvalue of the transition matrix.

\section{Preliminaries}

Given vectors $v, \mu \in \R^{N}$ such that $\mu$ has non-negative entries, we define
\[
\|v\|_{L_p(\mu)}^p = \sum_{i=1}^N |v_i|^p \mu_i.
\]
Additionally, we define $\|v\|_{\infty} = \max_{i} |v_i|$

We define the inner product for vectors $u, v \in \R^{N}$ and $\mu \in \R^N$ with positive entries as
\[
\langle u, v \rangle_{L_2(\mu)} = \sum_{i=1}^N u_i v_i \mu_i
\]

Given a matrix $M \in \R^{k \times k}$, we define the $L_2(\mu) \rightarrow L_2(\mu)$ operator norm to be
\[
\|M\|_{L_2(\mu) \rightarrow L_2(\mu)} = \sup_{x\st \|x\|_{L_2(\mu)} \leq 1} \| Mx \|_{L_2(\mu)}
\]

For simplicity, when $\mu$ is the all-$1$ vector, we will drop the $L_2(\mu)$ and refer to the inner product of $u$ and $v$ as just $\langle u, v \rangle$ and the operator norm of a matrix $M$ as just $\|M\|$. We will also refer to the norm of the vector $v$ as $\|v\|_2$.

For a stationary reversible Markov chain $\{Y_i\}_{i=1}^{\infty}$ with stationary distribution $\mu$, we define the averaging operator to be $E_{\mu} = \mathbf{1} \mu^*$ where $\mathbf{1}$ is the all-$1$ vector.
Let $A$ be the transition matrix of the Markov chain.
Note that $A$ and $E_{\mu}$ share the same top eigenvalue and eigenvector, and that $\left\|A-E_{\mu}\right\|_{L_2(\mu) \rightarrow L_2(\mu)} \leq 1$.

\subsection{Chaining}

We introduce various ideas from the theory of generic chaining that will be used in this note.
Let $T \subseteq \R^N$ be any set.
We first define the $\gamma_{\alpha}$ functional.
\begin{definition}
\[
\gamma_{\alpha}(T, d) = \inf \sup_{t \in T} \sum_{i=0}^{\infty} 2^{i/\alpha} \min_{t' \in T_i} d(t, t'),
\]
where the infimum is taken over all sequences of subsets $T_0 \subseteq T_1\subseteq \cdots \subseteq T$ such that $|T_0| = 1$ and $|T_i| \leq 2^{2^{i}}$ for $i \geq 1$.
\end{definition}

The following theorem, which appears as Theorem 4.5.13 in~\cite{T21}, allows one to obtain bounds on the supremum of a process that satisfies a certain increment condition.
We note that the tail bound can be found within the proof of Theorem 4.5.13 in~\cite{T21}.
\begin{theorem}[Theorem 4.5.13 in \cite{T21}]\label{thm:talbern}
Let $T$ be a set equipped with two distances $d_1$ and $d_2$, and let $(X_t)_{t \in T}$ be a centered process ($\E[X_t] = 0$ for all $t \in T$) which satisfies, for all $s, t \in T$ and all $u > 0$,
\[
\Pr\left[|X_s-X_t| \geq u\right] \leq 2 \exp\left(-\min \left(\frac{u^2}{d_2(s, t)^2}, \frac{u}{d_1(s, t)}\right)\right).
\]
Then,
\[
\E\left[\sup_{s, t \in T} |X_s - X_t|\right] \leq C_1 (\gamma_1(T, d_1)+\gamma_2(T, d_2)).
\]
and
\[
\Pr\left[\sup_{s, t \in T} |X_s - X_t| \geq  C_1 u^2 \gamma_1(T, d_1)+ C_1 u\gamma_2(T, d_2) \right] \leq C_2 \exp(-u^2)
\]
for some universal constants $C_1, C_2$.
\end{theorem}

The majorizing measures theorem, due to Talagrand~\cite{T87} (see also Theorem 2.10.1 in~\cite{T21})
gives tight bounds on the expected value of $\sup_{t \in T} \langle g, t \rangle$, where $g$ is a vector of independent Gaussians with mean $0$ and such that $g_i$ has variance $\mu_i$, in terms of $\gamma_2(T, L_2(\mu))$.
We state the theorem below.

\begin{theorem}[Talagrand's majorizing measures theorem]\label{thm:mm}
For some universal constant $C$, and for every symmetric $T \subseteq \R^n$,
\begin{equation*}\label{eq:mm}
\frac{1}{C}\gamma_2(T, L_2(\mu))\leq 
\mathbb{E}_{g_i \sim \mathcal{N}(0, \mu_i)}\left[\sup_{t \in T} \langle g, t \rangle\right]
\leq
{C}\gamma_2(T, L_2(\mu)).
\end{equation*}
\end{theorem}

% \begin{theorem}[Talagrand's majorizing measures theorem]\label{thm:mmlp}
% Let $y_1, \ldots, y_N$ be independent symmetric random variables satisfying $\Pr[|y_i| \geq t] = \exp(-t)$.
% Then, for some universal constant $C$, and for every symmetric $T \subseteq \R^N$,
% \begin{equation*}\label{eq:mm}
% \frac{1}{C}(\gamma_2(T, \ell_2)+\gamma_1(T, \ell_{\infty}))\leq 
% \mathbb{E}\left[\sup_{t \in T} \langle y, t \rangle\right]
% \leq
% {C}(\gamma_2(T, \ell_2)+\gamma_1(T, \ell_{\infty})).
% \end{equation*}
% \end{theorem}

\section{Proof of Theorems~\ref{thm:mainsame} and~\ref{cor:mstail}}

We start by recalling the following theorem due to Paulin~\cite{P15} which generalizes Bernstein's inequality to Markov chains.

\begin{theorem}[Theorem 3.3 in~\cite{P15}]\label{thm:ledoux}
Let $\{Y_i\}_{i=1}^{\infty}$ be a stationary reversible Markov chain with state space $[N]$, transition matrix $A$ and stationary measure $\mu$ so that $Y_1$ is distributed according to $\mu$.
Let $E_{\mu} = \mathbf{1} \mu^*$ and let $\lambda = \|A-E_{\mu}\|_{L_2(\mu) \rightarrow L_2(\mu)}$.

Let $f_1, \ldots, f_n: [N] \rightarrow [-M, M]$ be so that $\E[f(Y_i)] = 0$ for all $i$ and $\E[f(Y_1)^2]+\cdots+\E[f(Y_n)^2] = \sigma^2$.
Then for $u \geq 0$,
\begin{equation}\label{eq:ledoux}
\Pr\left[\left| \sum_{i=1}^{n} f_i(Y_i)\right| \geq u\right] \leq 2\exp\left(-\frac{u^2 (2(1-\lambda)-(1-\lambda)^2)}{8\sigma^2+20 u M}\right).
\end{equation}
\end{theorem}

% If we generalize to functions $f$ such that $f: [N] \rightarrow [-M, M]$, we obtain the bound
% \begin{equation}\label{eq:finallez}
% 2 e^{1/5}\exp\left(-\min\left\{\frac{u^2 (1-\lambda)}{4 \sigma^2}, \frac{u (1-\lambda)}{20 M}\right\}\right).
% \end{equation}

% Note that in Eq.~\eqref{eq:ledoux}, there is a factor of $2e^{(1-\lambda)/5}$ rather than $2$.
% Following the proof of Theorem~\ref{thm:talbern}, one can replace $C$ with $Ce^{(1-\lambda)/5}$ for our purposes.

% For $f, g \in \R^{N}$, define
% \[
% d_1(f, g) = \frac{20}{1-\lambda}\max_{v \in [N]} |f(v)-g(v)|
% \]
% and
% \[
% d_2(f, g)^2 = \frac{4n }{1-\lambda}\E_{\mu}\left[(f(v)-g(v))^2\right]
% \]

% Note that 
% \[
% d_1(f, g) =  O\left( \frac{\|f-g\|_{\infty}}{1-\lambda}\right)
% \text{ and }
% d_2(f, g) = O\left(\frac{\sqrt{n}\|f-g\|_{L_2(\mu)}}{\sqrt{1-\lambda}}\right).
% \]

% \begin{proof}
% Note that
% \[
% \frac{\gamma_1(T, \ell_{\infty})}{(1-\lambda)}+\frac{\sqrt{n} \gamma_2(T, L_2(\mu))}{\sqrt{1-\lambda}}
% \leq
% \frac{\sqrt{n}}{\sqrt{1-\lambda}}\left(\gamma_1(T, \ell_{\infty})+\gamma_2(T, L_2(\mu))\right)
% \]
% as long as $n \geq 1/(1-\lambda)$.
% \end{proof}

Combining Theorem~\ref{thm:talbern} and Theorem~\ref{thm:ledoux} yields the following corollary, which then gives a strategy for the proofs of Theorems~\ref{thm:mainsame} and~\ref{cor:mstail}.

\begin{corollary}\label{cor:same}
Let $\{Y_i\}_{i=1}^{\infty}$ be a stationary reversible Markov chain with state space $[N]$, transition matrix $A$ and stationary measure $\mu$ so that $Y_1$ is distributed according to $\mu$.
Let $\lambda = \|A-E_{\mu}\|_{L_2(\mu) \rightarrow L_2(\mu)}$.

Let $T \subseteq  \R^{n \times N}$ be any symmetric set of vectors containing $0$ such that $\E_{v \sim \mu}[t(i, v)] = 0$ for all $i$.
Let $W_t = \sum_{i=1}^{n} t(i, Y_i)$ for all $t \in T$.
Then,
\[
\E\left[\sup_{t \in T} |W_t|\right] \leq C_1 \left(\frac{\gamma_1(T, \ell_{\infty})}{1-\lambda}+\frac{\gamma_2(T, L_2((1, \ldots, 1) \otimes \mu))}{\sqrt{1-\lambda}}\right)
\]
and
\[
\Pr\left[\sup_{t \in T} |W_t| \geq u \right] \leq C_2 \exp\left(- C_1 (1-\lambda) \min\left\{\frac{u}{\gamma_1(T, \ell_{\infty})}, \frac{u^2}{\gamma_2(T, L_2((1, \ldots, 1) \otimes \mu))^2}\right\}\right)
\]
for some universal constant $C_1, C_2$.
\end{corollary}
\begin{proof}
Note that because $T$ contains $0$,
$\sup_{t \in T} |W_t-0| \leq \sup_{s, t \in T} |W_{s}-W_t|$.
By Theorem~\ref{thm:talbern}, it is enough to show that
\begin{equation}\label{eq:cor}
\Pr\left[|W_s-W_t| \geq u\right] \leq 2 \exp\left(-\min \left(\frac{u^2}{d_2(s, t)^2}, \frac{u}{d_1(s, t)}\right)\right).
\end{equation}
for
\[
d_1(s, t) =  O\left( \frac{\|s-t\|_{\infty}}{1-\lambda}\right)
\text{ and }
d_2(s, t) = O\left(\frac{\|s-t\|_{L_2((1, \ldots, 1) \otimes \mu)}}{\sqrt{1-\lambda}}\right).
\]
For $s, t \in \R^{n \times N}$, define
\[
d_1(s, t) = \frac{40}{1-\lambda}\max_{(i, v) \in [n] \times [N]} |s(i, v)-t(i, v)|
\]
and
\[
d_2(f, g)^2 = \frac{16}{1-\lambda}\sum_{i=1}^{n}\E_{\mu}\left[(s(i, v)-t(i, v))^2\right].
\]
Then Eq.~\eqref{eq:cor} follows from Theorem~\ref{thm:ledoux} and noting that the right-hand side of Eq.~\eqref{eq:ledoux} can be bounded above by
\[
2 \exp\left(-\min\left\{\frac{u^2(1-\lambda)}{16 \sigma^2}, \frac{u (1-\lambda)}{40 M}\right\}\right).
\]
\end{proof}

In particular, to bound $\E\left[\|f_1(Y_1)+\cdots+f_n(Y_n)\|\right]$ in Theorem~\ref{thm:mainsame}, it is enough to identify an appropriate set $T$ and bound $\gamma_1(T, \ell_{\infty})$ and $\gamma_2(T, L_2((1, \ldots, 1) \otimes \mu))$.
To bound $\gamma_2(T, L_2((1, \ldots, 1) \otimes \mu))$, one can apply Theorem~\ref{thm:mm}.
Finally, we prove Theorem~\ref{thm:mainsame}, by identifying a set $T$ that can be used to bound $\E\left[\|f_1(Y_1)+\cdots+f_n(Y_n)\|\right]$ and then bounding $\gamma_1(T, \ell_{\infty})$.

\begin{proof}[Proof of Theorems~\ref{thm:mainsame} and~\ref{cor:mstail}]
We start by constructing a set $T$ such that 
\[
\E\left[\|f_1(Y_1)+\cdots+f_n(Y_n)\|\right] = 
\E\left[\sup_{t \in T} |W_t|\right] 
\]
where $W_t = \sum_{i=1}^{n} t(i, Y_i)$.
This will allow us to apply Corollary~\ref{cor:same}.

Let $(X^*, \|\cdot\|_{*})$ be the dual space of $X$ with closed unit ball $B^*$.
Recall that
\begin{equation}\label{eq:norm}
\|x\| = \sup_{x^* \in B^*} |\langle x, x^*\rangle|.
\end{equation}
(see for instance, Theorem 4.3 in~\cite{R91}.)
For each $x^* \in B^*$, define the vector $t_{x^*}$ by $t_{x^*}(i, v) = \langle x^*, f_i(v) \rangle$ for all $i$ and $v$. Let $T$ be the set of all $t_{x^*}$ for all $x^* \in B^*$.
Then,
\[
\left\|f_1(Y_1)+\cdots+f_n(Y_n)\right\| = \sup_{x^* \in B^*} \langle x^*, f_1(Y_1)+\cdots+f_n(Y_n) \rangle = \sup_{t \in T} \sum_{i=1}^{n} t(i, Y_i).
\]

Note that because $\E[f_i(Y_i)] = 0$, it follows that $\E[t(i, Y_i)] = 0$ for all $i$.
The set $B^*$ is symmetric, and thus, for every $t_{x^*} \in T$, one has $-t_{x^*} = t_{-x^*} \in T$ and thus $T$ is symmetric.
Additionally, because $B^*$ contains $0$, so does $T$.
%Additionally $B^*$ is convex, and so for any $x_1^*, x_2^* \in B^*$ and $c \in [0, 1]$, one has $ct_{x_1^*}+(1-c)t_{x_2^*} = t_{cx_1^*+(1-2)x_2^*} \in T$.
Thus, Corollary~\ref{cor:same} applies here, and it is enough to bound $\gamma_2(T, L_2((1, \ldots, 1) \otimes \mu))$ and $\gamma_1(T, \ell_{\infty})$.

To bound $\gamma_2(T, L_2((1, \ldots, 1) \otimes \mu))$, we use Theorem~\ref{thm:mm}.
In particular, we have 
\[
\gamma_2(T, L_2((1, \ldots, 1) \otimes \mu))\leq 
C \mathbb{E}_{g_{j}^{(i)} \sim \mathcal{N}(0, \mu_j)}\left[\sup_{t \in T} \langle g, t \rangle\right]
=
C \E\left[\left\|\sum_{i=1}^{n}\sum_{v=1}^N g_{v}^{(i)} f_i(v)\right\|\right] 
\]
where the equality follows from the construction of $T$.

To bound $\gamma_1(T, \ell_{\infty})$, note that if $\left\|x^*_1-x^*_2\right\|_{*} \leq \eps$, then $\left\langle x^*_1-x^*_2, f_i(v) \right\rangle \leq \left\|x^*_1-x^*_2\right\|_{*} \|f_i(v)\|$ by Eq~\eqref{eq:norm}.
Thus, $\left\|t_{x^*_1}-t_{x^*_2}\right\|_{{\infty}} \leq \left\|x^*_1-x^*_2\right\|_{*}$

The above demonstrates that there is a correspondence of partitions of $T$ to partitions of the unit ball of $(X^*, \|\cdot\|_*)$.
Thus, we will consider partitions of the latter which will correspond to partitions of the former.

Let $T_i$ correspond to any maximal set $S$ of points from $B^*$ such that every pair of points are a distance at least $\eps$ from each other.
Because the set is maximal, every point in $B^*$ is at most $\eps$ away from some point in $S$.
On the other hand, because every pair of points in $S$ are a $\eps$ distance away from each other, the set of balls of radius $\eps/2$ centered at points in $S$ are disjoint and are contained inside a ball of radius $1+\eps/2$.
Thus, there are at most $(2/\eps+1)^k$ points, which is at most $2^{2^i}$ as long as $\eps \leq C 2^{-2^i/k}$ for some constant $C$.

Then for constants $C_1, C_2$, and $C_3$.
\begin{align*}
\gamma_{1}(T, \ell_{\infty})
&\leq
\gamma_{1}(B^*, *) \\
&=  \inf \sup_{x_1^* \in B^*} \sum_{i = 0}^{\infty} 2^{i} \min_{x_2^* \in B^*} \left\|t_{x^*_1}-t_{x^*_2}\right\|_{{\infty}}\\
&=  \inf \sup_{x_1* \in B^*} \sum_{i = 0}^{\infty} 2^{i} \min_{x_2^* \in B^*} \left\|x^*_1-x^*_2\right\|_{*}\\
&\leq C_1\sum_{i = 0}^{\infty} 2^{i} 2^{-2^i/k} \\
&\leq C_1\sum_{i = 0}^{\log(k)} 2^{i}+C_2 \sum_{i=\log(k)+1}^{\infty} 2^{i} 2^{-2^i/k} \\
&\leq C_3 k
\end{align*}
as desired.
\end{proof}

\subsection{Sharper versions of Theorem~\ref{thm:mainsame}}\label{sec:sharp}

The extra $Ck$ term in the bound of Theorem~\ref{thm:mainsame} is a result of the use of partitions of the dual space of the Banach space we are studying to bound $\gamma_{1}(T, \ell_{\infty})$.
In certain cases, one can replace the dual space with some other lower-dimensional space to obtain a slightly sharper estimate.
We do this for matrices under the operator norm, and for vectors under the $\ell_{\infty}$ norm.

We stress that in both cases, the sharper bound we obtain does not match what is known from~\cite{GLSS18} and~\cite{NSW24} in the case of matrices under the operator norm, or from the application of the union bound to say Theorem~\ref{thm:ledoux} in the case of vectors under the $\ell_{\infty}$ norm.
In both cases, the $Ck$ term can be removed, or equivalently replaced by a $C \sqrt{\log(k)}$ term.
The goal of this section is not to prove optimal bounds, but rather, to demonstrate how it might be possible to improve on the bounds in Theorem~\ref{thm:mainsame} when considering a specific Banach space.

We start with the sharper bound for matrices under the operator norm, using the fact that the unit ball of the dual space is the set of convex combinations of a set of rank-$1$ matrices.

\begin{theorem}\label{thm:intdiff}
When the Banach space in Theorem~\ref{thm:mainsame} is the set of $d \times d$ matrices equipped with the $\ell_2 \rightarrow \ell_2$ operator norm, for some constant $C$,
\[
\E\left[\|f_1(Y_1)+\cdots+f_n(Y_n)\|\right] \leq \frac{C d}{1-\lambda}+\frac{C}{\sqrt{1-\lambda}} \E\left[\left\|\sum_{i=1}^{n} \sum_{v=1}^N g_v^{(i)} f_i(v)\right\|\right] 
\]
\end{theorem}
\begin{proof}
As in the proof of Theorem~\ref{thm:mainsame}, we construct a set $T$ such that 
\[
\E\left[\|f_1(Y_1)+\cdots+f_n(Y_n)\|\right] = 
\E\left[\sup_{t \in T} |W_t|\right] 
\]
where $W_t = \sum_{i=1}^{n} t(i, Y_i)$.
This will allow us to apply Corollary~\ref{cor:same}.

Let $B^{d} \subseteq \R^{k}$ be the unit ball under the $\ell_2$ norm.
For each $x, y \in S^{k-1}$, define the vector $t_{x,y}(i, v)$ by $t_{x, y}(i, v) = \left\langle x, f_i(v) y \right\rangle$, and let $T$ be the set of all $t_{x, y}$ for all $x, y \in B^{d}$.
Then,
\[
\left\|\sum_{i=1}^{n}f_i(Y_{i})\right\|  =
\sup_{x, y \in B^{d}} \left\langle x, \left(\sum_{i=1}^{n}f_i(Y_{i})\right) y \right\rangle =
\sup_{t \in T} \sum_{i=1}^{n} t(i, Y_{i}).
\]

Similarly as in the proof of Theorems~\ref{thm:mainsame} and~\ref{cor:mstail}, note that $B^{d}$ is symmetric and contains $0$, and thus $T$ is symmetric and contains $0$.
Additionally, $\E_{v \sim \mu}[t(i, v)] = 0$ for all $i$ and $t \in T$ as $\E_{v \sim \mu}[f_i(v)] = 0$ for all $i$.
Thus, Corollary~\ref{cor:same} applies here  and it is enough to bound $\gamma_2(T, L_2((1, \ldots, 1) \otimes \mu))$ and $\gamma_1(T, \ell_{\infty})$.

As before, to bound $\gamma_2(T, L_2((1, \ldots, 1) \otimes \mu))$, we use Theorem~\ref{thm:mm}.
In particular, we have 
\[
\gamma_2(T, L_2((1, \ldots, 1) \otimes \mu))\leq 
C \mathbb{E}_{g_{j}^{(i)} \sim \mathcal{N}(0, \mu_j)}\left[\sup_{t \in T} \langle g, t \rangle\right]
=
C \E\left[\left\|\sum_{i=1}^{n}\sum_{v=1}^N g_{v}^{(i)} f_i(v)\right\|\right] 
\]

To bound $\gamma_1(T, \ell_{\infty})$, note that 
\begin{align*}
\left|\langle x, f(v) y \rangle - \langle x^*, f(v) y^* \rangle\right|
&=
\left|\langle x, f(v) (y-y^*) \rangle +\langle (x-x^*), f(v) y^* \rangle\right| \\
&\leq \left|\langle x, f(v) (y-y^*) \rangle\right| +\left|\langle (x-x^*), f(v) y^* \rangle\right| \\
&\leq
\|x-x^*\|_{2}+\|y-y^*\|_{2}
\end{align*}
where the last inequality follows from the Cauchy-Schwarz inequality and the fact that $\|f(v)\| \leq 1$.
Thus, $\|t_{x, y} - t_{x^*, y^*}\|_{\infty} \leq \|x-x^*\|_{2}+\|y-y^*\|_{2}$.

Finally, one obtains the bound $\gamma_1(T, \ell_{\infty}) \leq C d$ using the same proof as used in the end of the proof of Theorem~\ref{thm:mainsame}.
In particular, there is a correspondence between partitions of $T$ and partitions of the set $B^{d} \times B^{d}$.
As in Theorem~\ref{thm:mainsame}, one can use the geometry of the latter to bound $\gamma_1(T, \ell_{\infty})$
\end{proof}

The following is a sharper bound of vectors under the $\ell_{\infty}$ norm, which uses the fact that the dual space is the set of convex combinations of a finite set of points.

\begin{theorem}\label{thm:intdiff}
When the Banach space in Theorem~\ref{thm:mainsame} is the set of $k$-dimensional vectors with the $\ell_{\infty}$ norm, for some constant $C$,
\[
\E\left[\|f_1(Y_1)+\cdots+f_n(Y_n)\|\right] \leq \frac{C \log(k)}{1-\lambda}+\frac{C}{\sqrt{1-\lambda}} \E\left[\left\|\sum_{i=1}^{n} \sum_{v=1}^N g_v^{(i)} f_i(v)\right\|\right] 
\]
\end{theorem}
\begin{proof}
As in the proof of Theorem~\ref{thm:mainsame}, we construct a set $T$ such that 
\[
\E\left[\|f_1(Y_1)+\cdots+f_n(Y_n)\|\right] = 
\E\left[\sup_{t \in T} |W_t|\right] 
\]
where $W_t = \sum_{i=1}^{n} t(i, Y_i)$.
This will allow us to apply Corollary~\ref{cor:same}.

Let $V$ be the set of vectors with exactly one non-zero coordinate that is either $1$ or $-1$, along with the vector $0$.
Note that $V$ contains $2k+1$ vectors.
For each $x \in V$, define $t_{x}(i, v) = \langle x, f_i(v) \rangle$, and let $T$ be the set of all $t_{x}$ for all $x \in V$.
Then,
\[
\left\|\sum_{i=1}^{n}f_i(Y_{i})\right\|  =
\sup_{x \in V} \left\langle x, \left(\sum_{i=1}^{n}f_i(Y_{i})\right) \right\rangle =
\sup_{t \in T} \sum_{i=1}^{n} t(i, Y_{i}).
\]

Note that $\E_{v \sim \mu}[t(i, v)] = 0$ for all $i$ and $t \in T$ as $\E_{v \sim \mu}[f_i(v)] = 0$ for all $i$.
Additionally, $V$ is symmetric and contains $0$, and thus $T$ is symmetric and contains $0$.
Thus, Corollary~\ref{cor:same} applies here, and it is enough to bound $\gamma_2(T, L_2((1, \ldots, 1) \otimes \mu))$ and $\gamma_1(T, \ell_{\infty})$.

As before, to bound $\gamma_2(T, L_2((1, \ldots, 1) \otimes \mu))$, we use Theorem~\ref{thm:mm}.
In particular, we have 
\[
\gamma_2(T, L_2((1, \ldots, 1) \otimes \mu))\leq 
C \mathbb{E}_{g_{j}^{(i)} \sim \mathcal{N}(0, \mu_j)}\left[\sup_{t \in T} \langle g, t \rangle\right]
=
C \E\left[\left\|\sum_{i=1}^{n}\sum_{v=1}^N g_{v}^{(i)} f_i(v)\right\|\right] 
\]

To bound $\gamma_1(T, \ell_{\infty})$, one can use the following sequence of subsets: $T_i = {0}$ if $i \leq \log\log(2k+1)$, and $T_i = T$ otherwise.
Note that $\|t-0\|_{\infty} \leq 1$ for all $t \in T$ by construction.
Then,
\[
\gamma_1(T, \ell_{\infty}) \leq \sum_{i=0}^{\log\log(2k+1)} 2^i \leq C\log(k).
\]
\end{proof}

\bibliographystyle{alphaabbrv}
\bibliography{expanderhoeffding}
\end{document}